\documentclass[leqno]{amsart}
\usepackage{amsmath}
\usepackage{mathtools}
\usepackage{amssymb}
\usepackage{amsthm}
\usepackage{graphicx}
\usepackage{enumerate}
\usepackage[mathscr]{eucal}
\theoremstyle{plain}
\newtheorem{theorem}{Theorem}[section]

\newtheorem{lemma}{Lemma}[section]

\theoremstyle{definition}
\newtheorem{definition}{Definition}[section]

\newtheorem{example}{Example}[theorem]

\numberwithin{equation}{section}
\usepackage[pagewise]{lineno}

\begin{document}

\title[Extreme contractions on Banach spaces]{Extreme contractions on finite-dimensional polygonal Banach spaces-II}
\author[Ray, Roy, Bagchi and Sain]{Anubhab Ray, Saikat Roy, Satya Bagchi, Debmalya Sain}

\newcommand{\acr}{\newline\indent}

\address[Ray]{Department of Mathematics\\ Jadavpur University\\ Kolkata 700032\\ West Bengal\\ INDIA}
\email{anubhab.jumath@gmail.com}

\address[Roy]{Department of Mathematics\\ National Institute of Technology Durgapur\\ Durgapur 713209\\ West Bengal\\ INDIA}
\email{saikatroy.cu@gmail.com}

\address[Bagchi]{Department of Mathematics\\ National Institute of Technology Durgapur\\ Durgapur 713209\\ West Bengal\\ INDIA}
\email{satya.bagchi@maths.nitdgp.ac.in}

\address[Sain]{Department of Mathematics\\ Indian Institute of Science\\ Bengaluru 560012\\ Karnataka \\India\\ }
\email{saindebmalya@gmail.com}

\thanks{The research of Anubhab Ray is supported by DST Inspire in terms of doctoral fellowship under the supervision of Prof. Kallol Paul. The research of Saikat Roy is supported by CSIR Junior research fellowship. The research of Dr. Debmalya Sain is sponsored by Dr. D. S. Kothari Postdoctoral Fellowship under the mentorship of Professor Gadadhar Misra. Dr. Sain feels elated to acknowledge the wonderful hospitality of his childhood friend Mr. Subhro Jana and his wife Mrs. Poulami Mallik.}


\begin{abstract}
We introduce the concept of weak L-P property for a pair of Banach spaces, in the study of extreme contractions. We give examples of pairs of Banach spaces (not) satisfying weak L-P property and apply the concept to compute the exact number of extreme contractions between a particular pair of polygonal Banach spaces. We also study the optimality of our results on the newly introduced weak L-P property for a pair of Banach spaces, by considering appropriate examples.  
\end{abstract}

\maketitle

\section{Introduction.}

The purpose of the present article is to study extreme contractions on finite-dimensional polygonal Banach spaces, in light of some previously  obtained results in a recent article \cite{SRP}. In contrast to the Hilbert space case \cite{Ga,K,N,S}, characterization of extreme contractions on Banach spaces is a difficult problem even in the finite-dimensional case, that remains vastly unexplored. We refer the readers to \cite{CM, G,I,Ki,L,LP,Sh,Sha} for some of the prominent research works in this direction that also illustrate the difficulty in building a general theory for characterizing extreme contractions on Banach spaces. The existence of such a general theory has been explored very recently in \cite{SPM} with emphasis on the two-dimensional case, by introducing the concept of compatible point pair (CPP). On the other hand, the work carried out in \cite{SRP} suggests that the aforesaid study has deep connections with the extremal structure of the unit balls of the concerned spaces. Our present article further explores this idea and establishes some interesting connections between extreme contractions and the extreme points of the unit balls of the domain space and the range space.\\
\\
In this article, letters $X$ and $Y$ denote real Banach spaces. Given a subset $ U $ of $ X, $ let $ |U| $ denote the cardinality of $ U. $ Let $B_X=\{x\in X: \|x\|\leq 1\}$ and $S_X=\{x\in X:\|x\|=1\}$ denote the unit ball and the unit sphere of $X$ respectively. Let $E_X$ be the set of all extreme points of the unit ball $B_X$. We say that $ X $ is polygonal if $ E_X $ is finite. Let $L(X,Y)$ be the Banach space of all bounded linear operators from $X$ to $Y$ endowed with usual operator norm. An operator $ T \in L(X,Y) $ is said to be an extreme contraction if $ T $ is an extreme point of the unit ball $ B_{L(X,Y)}. $ For a bounded linear operator $T\in L(X,Y), $ let $ M_T $ denote the norm attainment set of $ T, $ i.e., $M_T=\{x\in S_X:\|Tx\|=\|T\|\}.$ Motivated by the work done in the seminal article \cite{LP} on finite-dimensional extreme contractions, the following definition was introduced in \cite{SRP}.

\begin{definition}
Let $ X,~Y $ be Banach spaces. We say that the pair $ (X,Y) $ has L-P (abbreviated form of Lindenstrauss-Perles) property if a norm one bounded linear operator $ T \in L(X,Y) $ is an extreme contraction if and only if $ T(E_X) \subseteq E_Y. $ 
\end{definition}

\noindent For further continuation of the study of extreme contractions on finite-dimensional Banach spaces, we introduce the following definition:

\begin{definition}
Let $ X,~Y $ be Banach spaces. We say that the pair $ (X,Y) $ has  weak L-P  property if  for any extreme contraction $ T \in L(X,Y), $ we have that $ T(E_X) \cap E_Y \neq \phi. $
\end{definition}

We illustrate the utility of the above definition in the study of extreme contractions on finite-dimensional polygonal Banach spaces. We furnish examples of several pairs of polygonal Banach spaces that satisfy weak L-P  property. In particular, it follows from our results in the present article that weak L-P property for a pair of Banach spaces depends on the extremal structure of the unit balls of the domain space and the range space. Moreover, as a concrete application of the concept of  weak L-P  property of a pair of Banach spaces, we explicitly compute the exact number of extreme contractions in a particular case involving two-dimensional polygonal Banach spaces. Our results in the present article further underline the pivotal role of extreme points of the domain space and the range space in the study of extreme contractions.

\section{Main Results}
We begin this section with the remark that Lima \cite[Lemma 3.2]{L} has explicitly constructed an extreme contraction in $ L(l_{\infty}^{4},l_{1}^{4}) $ under which the image of any extreme point of the unit ball of the domain space $ l_{\infty}^{4} $ is not an extreme point of the unit ball of the range space $ l_{1}^{4}. $ In particular, this proves that the pair $ (l_{\infty}^{4},l_{1}^{4}) $ does not have weak L-P property. Our first aim is to obtain some nontrivial examples of pairs of polygonal Banach spaces that has weak L-P property. As we will see next, the number of extreme points of the unit ball of the domain space is of paramount importance in this  study. 

\begin{theorem}\label{th1}
Let $X$ be an $n-$dimensional real polygonal Banach space such that $B_X$ has exactly $2n+2$ extreme points. Let $Y=l_\infty^m,$ where $m\leq n.$ Then the pair $ (X,Y) $ has weak L-P property.
\end{theorem} 

\begin{proof}
Let $E_X=\{\pm x_1, \pm x_2, \cdots, \pm x_{n+1}\}.$ Let $T\in L(X,Y)$ be an extreme contraction. It follows from Theorem 2.2 of \cite{SRP} that $span(M_T\cap E_X)=X $ and moreover, if $|M_T\cap E_X|=2n$ then $T(M_T \cap E_X)\subseteq E_Y.$ Therefore, in this case we have nothing more to show. Let us assume that $E_X \subseteq M_T.$ We claim that $T(E_X)\cap E_Y\neq \phi.$ If possible, suppose that $T(E_X)\cap E_Y= \phi.$ Without any loss of generality, we may assume that $\{x_1, x_2, \cdots, x_n\}$ is a basis of $X.$ So there exist scalars $\alpha_1,\alpha_2,\cdots ,\alpha_n,$ with at least two of them non-zero, such that $x_{n+1}=\alpha_1x_1+\alpha_2x_2+\cdots +\alpha_nx_n.$ Let for each $ i \in \{1,2, \ldots, n\}, $ $T(x_i)=(a_{i1}, a_{i2}, \cdots, a_{im}).$\\ 
Therefore, $T(x_{n+1})=\alpha_1T(x_1)+\alpha_2T(x_2)+\cdots +\alpha_nT(x_n)=(b_1,b_2,\cdots ,b_m)~~~~\mathrm{(say)}.$
Clearly, $E_Y=\{(y_1,y_2, \cdots ,y_m):y_j\in\{+1,-1\}, 1\leq j\leq m\}. $ So for each $1\leq i\leq n$, there exists at least one $j,$ $1\leq j\leq m,$ such that $|a_{ij}|<1.$ Let us define $ G_i=\big\{ j\in \{ 1, 2, \dots , m \} : |a_{ij}|<1 \big\}.$
Clearly, $G_i\neq \phi$ for each $i.$ \\

Now we complete the proof of the theorem by considering the following two cases:\\

\noindent \textbf{Case 1:}
Suppose there exist distinct $p,q\in \{1,2, \cdots, n\}$ such that $G_p\cap G_q\neq \phi.$ Let $t\in G_p\cap G_q.$ We note that such a pair $\{p,q\}$ always occurs whenever $m < n.$ We define two linear operators $T_1,T_2 ~:~X\rightarrow Y $ by \\

$T_1(x_i)=T(x_i)\, \forall~~ i\neq p,q,$ \,\hfil \, $T_2(x_i)=T(x_i)\, \forall~~ i\neq p,q, $
$$T_1(x_p)=(a_{p1}, \cdots, a_{pt}-\epsilon_1,\cdots, a_{pm}), \,\,\,\,\, T_2(x_p)=(a_{p1},\cdots, a_{pt}+\epsilon_1,\cdots,  a_{pm}),$$
$$T_1(x_q)=(a_{q1}, \cdots, a_{qt}-\epsilon_2,\cdots, a_{qm}), \,\,\,\,\, T_2(x_q)=(a_{q1}, \cdots, a_{qt}+\epsilon_2,\cdots, a_{qm}),$$

where  $\epsilon_1$ and $\epsilon_2$ are chosen in such a way that $|a_{pt}\pm \epsilon_1|<1, |a_{qt}\pm\epsilon_2|<1$ and $\alpha_{p}\epsilon_1+\alpha_q\epsilon_2=0.$
Then $T_1(x_{n+1})=T_2(x_{n+1})=T(x_{n+1}),$ $T=\frac{1}{2}(T_1+T_2)$ and both $\|T_1\|,\|T_2\|=1.$ In addition, $T_1$ and $T_2$ are different from $T.$ Clearly, this contradicts that $T$ is an extreme contraction.\\

\noindent \textbf{Case 2:}
Suppose $G_p\cap G_q = \phi,$ for every $p,q\in \{1,2,\cdots ,n\}$ with $p\neq q.$ We should note that this case can only happen for $m=n.$ 
Since, $T(E_X)\cap E_Y=\phi,$ there exists some $r,$ $1\leq r\leq m,$ such that $|b_r|<1.$ Also there exists some $s,$ $1\leq s\leq n,$ such that $r\in G_s.$ We choose $\epsilon_r$ in such a way that $|a_{sr}\pm \frac{\epsilon_r}{\alpha_r}|<1$ (provided $\alpha_r\neq 0$) and $|b_r \pm \epsilon_r|<1.$ If $\alpha_r=0$ then we choose $\epsilon_r$ such that $|a_{sr}\pm \epsilon_r|<1.$ Now, define two linear operators $ T_1,T_2 ~:~ X \rightarrow Y$ by\\

$T_1(x_i)=T(x_i)\, \forall~~ i\neq s,$ \,\hfil  \, $T_2(x_i)=T(x_i)\, \forall~~ i\neq s,$ 
$$T_1(x_s)=(a_{s1}, \cdots, a_{sr}-\frac{\epsilon_r}{\alpha_r},\cdots, a_{sm}),\,\,\,\,\, T_2(x_s)=(a_{s1}, \cdots, a_{sr}+\frac{\epsilon_r}{\alpha_r},\cdots, a_{sm}).$$

Then,  $T=\frac{1}{2}(T_1+T_2)$ and $\|T_1\|=\|T_2\|=1.$ In addition, $T_1$ and $T_2$ are different from $T.$ This is a contradiction to the fact that $T$ is an extreme contraction. In other words, whenever $ T \in L(X,Y) $ is an extreme contraction, we have that $ T(E_{X}) \cap E_{Y} \neq \phi. $ This completes the proof of the fact that the pair $ (X,Y) $ has weak L-P property and establishes the theorem.
\end{proof}

In the next example we illustrate that it is not possible to improve the condition $ m \leq n $ in the above theorem.

\begin{example} \label{ex1}
Let $ X $ be the two dimensional real Banach space whose unit sphere is a regular hexagon with the vertices $ \pm x_1 = \pm(1,0), \pm x_2 = \pm (\frac{1}{2},\frac{\sqrt 3}{2}), \pm x_3 = \pm(-\frac{1}{2},\frac{\sqrt 3}{2}) $. Let $Y=l_{\infty}^3.$ Clearly, $ \{ x_1, x_2 \} $ forms a basis of $ X $ and $ x_3 =x_2-x_1. $ Now, we define a linear operator $T:X\rightarrow Y$ in the following way:

$$ T(x_1)=(1,-1,0),\,\,\,\ T(x_2)=(1,0,1). $$

Therefore, $ T(x_3)=(0,1,1).$ We claim that $T$ is an extreme contraction. If not, then there exists $T_1,T_2\in S_{L(X,Y)}$ such that $T\neq T_1,T_2$ and $T=\frac{1}{2}(T_1+T_2).$ Let 
$$T_1(x_1)=(1+\epsilon_{11},-1+\epsilon_{12},\epsilon_{13}),\,\,\,\,T_1(x_2)=(1+\epsilon_{21},\epsilon_{22},1+\epsilon_{23}),$$
where $\epsilon_{ij} $ are arbitrary real numbers for all $i\in\{1,2\}$ and for all $j\in\{1,2,3\}.$
Since $T=\frac{1}{2}(T_1+T_2),$ we must have,
$$ T_2(x_1)=(1-\epsilon_{11},-1-\epsilon_{12},-\epsilon_{13}),\,\,\,T_2(x_2)=(1-\epsilon_{21},-\epsilon_{22},1-\epsilon_{23}) .$$
 Now, as $ T_1, T_2 \in S_{L(X,Y)}, $  it is easy to see that $\epsilon_{11}=\epsilon_{12}=\epsilon_{21}=\epsilon_{23}=0. $ So we have, $ T_1(x_3)=(0,1+\epsilon_{22},1-\epsilon_{13}) $ and $ T_2(x_3)=(0,1-\epsilon_{22},0,1+\epsilon_{13}). $ Moreover, if any of the $\epsilon_{13}$ or $\epsilon_{22}$ is non-zero, then either $\|T_1(x_3)\|$ or $\|T_2(x_3)\|$  becomes greater than $1.$ Therefore we must have $\epsilon_{13}=\epsilon_{22}=0.$ Thus, we have that $T_1=T_2=T. $ However, this is clearly a contradiction to our assumption that $T\neq T_1,T_2.$ Hence, $T$ is an extreme contraction although $T(E_X)\cap E_Y=\phi.$ Therefore, the pair $ (X,Y) $ does not have weak L-P property.
\end{example}

In view of Theorem \ref{th1}, it is natural to ask whether $ l_{\infty}^{m} $ can be replaced by other well-known polygonal Banach spaces, say, $ l_{1}^{m}. $ We will prove a positive result in this direction. Let us first fix the following notation in order to proceed further:\\

\begin{definition} For a fixed index $ i, $ let $ v_i=(a_{i1}, a_{i2}, \dots , a_{im}) \in \mathbb{R}^{m} $ be an $ m $-tuple. Let us define $ C_i=\{ j\in \{1,2, \dots, m \} : a_{ij}\neq 0 \} $. We would like to  observe that if $ v_i $ is a non-extreme point on the unit sphere of $ l_1^{ m } $, then $ |C_i|\geq 2 $.
\end{definition}

We require the following lemma in order to prove the desired result. 

\begin{lemma} \label{combinatorics}
Let $ \{ v_i: i\in \{1, 2, \dots, k \} \} $ be a collection of non-extreme points on the unit sphere of $ l_1^{ m } $, where $ m\geq 2 $.\\

(i) If $ k>m(m-1) $ then there exists a triplet of distinct numbers $ \{ r,s,t \} \subseteq \{ 1, 2, \dots , k \} $ such that $ | C_r\cap C_s\cap C_t |\geq 2 $.\\

(ii) If $ k=m(m-1) $ and for any triplet of distinct numbers $ \{ r, s, t \} \subseteq \{ 1, 2, \dots , k \} $, $ | C_r\cap C_s\cap C_t|\leq 1 $ then $ | C_i |=2 $ for all $ i\in \{ 1, 2, \dots , k \} $. In this case, for any pair of distinct numbers $  \alpha, \beta \in \{ 1, 2, \dots , m \}, $ there exist distinct $ p,q \in \{ 1, 2, \dots , k \} $ such that $  C_p\cap C_q = \{ \alpha, \beta \} $.
\end{lemma}

\begin{proof}
If $ m=2 $ then it is easy to see that both $ (i) $ and $ (ii) $ are trivially true. Without any loss of generality, let us assume that $ m\geq 3.$\\

$ (i) $ If possible, suppose that for any triplet of distinct numbers $ \{ r, s, t\} \subseteq \{ 1, 2, \dots , k \} $, we have that $ | C_r\cap C_s\cap C_t|\leq 1 $. Without any loss of generality, suppose $ |C_1| =  \max  \{ |C_i| : 1\leq i \leq k \} = u $. Define $ \mathcal{A} = \{ v_i : C_i \subseteq C_1 \} $ and $ \mathcal{B} = \{ v_i : C_i \not\subset C_1 \} $. It is immediate that $ \mathcal{A} \cap \mathcal{B} = \phi $ and therefore $ k = | \mathcal{A} | + | \mathcal{B} | $. According to our assumption, for any triplet of distinct numbers $ \{ r, s, t \} \subseteq \{ 1, 2, \dots , k \} $, we have that $ | C_r\cap C_s\cap C_t|\leq 1 $. It is easy to observe that $ | \mathcal{A} | \leq \{ { u \choose 2 } +1 \} $. Now, if $ u=m $ then $ \mathcal{B} = \phi. $ Hence, we have $ k < m(m-1). $ Also, when $ u = m-1 $, $ | \mathcal{B} | \leq 2(m-1). $ Therefore, in that case we again have that $ k \leq m(m-1) $ and equality holds if $ m= 3. $ Finally, if $ u \leq (m-2) $ then it is easy to deduce that $ | \mathcal{B} | \leq 2\{ { m - u \choose 2 } + { u \choose 1 } { m-u \choose  1 }  \} $. Consequently $ k \leq m(m-1) $ and equality holds if $ u = 2.  $ In each of these cases, we arrive at a contradiction to our hypothesis that $ k > m(m-1). $ This completes the proof of $ (i) $.\\

$ (ii) $ If possible, suppose that $ | C_i | \neq 2 $ for some $ i\in \{ 1, 2, \dots , k \} $. Without any loss of generality, we may assume that $ |C_1| =  \max  \{ |C_i| : 1\leq i \leq k \} = u \geq 3 $. Then by analogous arguments as given in the proof of $ (i) $, we can conclude that $ k < m(m-1) $. This contradicts the fact that $k = m(m-1) $. Therefore $ | C_i |=2 $ for all $ i \in \{ 1, 2, \dots , k \} $. Now, $ k = m(m-1) $ vectors are consistent with the condition that for any triplet of distinct numbers $ \{ r, s, t \} \subseteq \{ 1, 2, \dots , k \} $, we have that $ | C_r\cap C_s\cap C_t|\leq 1 $. Therefore, for each pair of distinct numbers $ \alpha, \beta  \in \{ 1, 2, \dots , m \}, $  there must exist exactly two vectors $ v_p, v_q \in S_{l_1^{m}} $ among the $ m(m-1) $ vectors such that $  C_p\cap C_q  = \{ \alpha, \beta \}. $ This completes the proof of $ (ii) $.
\end{proof}

As an immediate application of Lemma \ref{combinatorics}, we next obtain another class of pairs of Banach spaces that has weak L-P property.

\begin{theorem}\label{th2}
Let $X$ be an $n-$dimensional real polygonal Banach space such that $B_X$ has exactly $2n+2$ extreme points. Let $Y=l_1^m,$ where $m(m-1) \leq n.$ Then the pair $ (X,Y) $ has weak L-P property. 
\end{theorem}

\begin{proof}
Let $E_X=\{\pm x_1, \pm x_2, \cdots, \pm x_{n+1}\}.$ Let $T\in L(X,Y)$ be an extreme contraction. It follows from Theorem 2.2 of \cite{SRP} that $span(M_T\cap E_X)=X$ and moreover, if $|M_T\cap E_X|=2n$ then $T(M_T \cap E_X)\subseteq E_Y.$ Therefore, in this case we have nothing more to show. Let us assume that $E_X \subseteq M_T. $ We claim that $T(E_X)\cap E_Y\neq \phi.$ If possible, suppose that $T(E_X)\cap E_Y= \phi.$ Without any loss of generality, we may assume that $\{x_1, x_2, \cdots, x_n\}$ is a basis of $X.$ So there exist scalars $\alpha_1,\alpha_2,\ldots ,\alpha_n,$ with at least two of them non-zero, such that $x_{n+1}=\alpha_1x_1+\alpha_2x_2+\cdots +\alpha_nx_n.$ Let for each $ i \in \{1,2,\ldots,n\},$ $T(x_i)=(a_{i1}, a_{i2}, \ldots, a_{1m}).$\\
Therefore, $T(x_{n+1})=\alpha_1T(x_1)+\alpha_2T(x_2)+\cdots +\alpha_nT(x_n)=(b_1,b_2,\cdots ,b_m)\,\,\,\mathrm{(say)}.$
Clearly, $E_Y=\{\pm e_j=(\underbrace{0,0,\cdots,\pm 1}_{j},\cdots,0): 1\leq j\leq m\}.$ Now for each $ i\in \{1,2,\cdots ,n\} $, $ T(x_i) $ are non-extreme points on the unit sphere of $ l_1^{m} $. Therefore $ |C_i|\geq 2$ for each $ i\in \{1,2,\cdots ,n\}, $ where $ C_i = \{ j\in \{ 1, 2, \dots , m \} : a_{ij}\neq 0\} $.\\

Now we complete the proof of the theorem by considering the following two cases:\\

\noindent \textbf{Case 1:}
Let $ n>m(m-1) $. Then by Lemma \ref{combinatorics}, there exists a triplet of distinct numbers $ \{ r, s, t \} \subseteq \{ 1, 2, \dots , n \} $ such that $ | C_r\cap C_s\cap C_t |\geq 2 $. Without any loss of generality, we assume $\{1,2\}\subseteq C_1\cap C_2\cap C_3.$
Let $\sigma=a_{11}a_{12},$ $\mu=a_{21}a_{22}, $ $\nu= a_{31}a_{32}.$ Clearly, $\sigma,\mu,\nu \neq 0 $ and at least  two of them have the same sign. Without any loss of generality, we assume $\sigma,\mu > 0.$ Choose non-zero real numbers $\epsilon, \delta $ in such a way that the following three conditions are satisfied:

$  (i) $  $ \alpha_1\epsilon+\alpha_2\delta=0, $
$$ (ii)~~\|(a_{11}-\epsilon, a_{12}+\epsilon,\cdots ,a_{1m})\|=\|(a_{11}+\epsilon, a_{12}-\epsilon,\cdots ,a_{1m})\|=\|T(x_1)\|, $$
$$(iii)~~\|(a_{21}-\delta, a_{22}+\delta,\cdots ,a_{2m})\|=\|(a_{21}+\delta, a_{22}-\delta,\cdots ,a_{2m})\|=\|T(x_2)\|.$$\\
We define two linear operators $T_1,T_2 ~:~ X\rightarrow Y $ by \\
$$T_1(x_1)=(a_{11}-\epsilon, a_{12}+\epsilon,\cdots ,a_{1m}),\,\,\,\,\,\,\,\,\,\,\,\,T_2(x_1)=(a_{11}+\epsilon, a_{12}-\epsilon,\cdots ,a_{1m}),$$
$$T_1(x_2)=(a_{21}-\delta, a_{22}+\delta,\cdots ,a_{2m}),\,\,\,\,\,\,\,\,\,\,\,\,\,T_2(x_2)=(a_{21}+\delta, a_{22}-\delta,\cdots ,a_{2m}), $$
$T_1(x_i)=T(x_i)\, \forall~~ i\in \{ 3,\cdots,n\}, $ \, \hfil \,  $ T_2(x_i)=T(x_i)\, \forall~~ i\in \{3,\cdots,n\}. $\\

Then $T_1(x_{n+1})=T_2(x_{n+1})=T(x_{n+1})$, $T=\frac{1}{2}(T_1+T_2),$ $\|T_1\|=\|T_2\|=1$. In addition $T_1,T_2 \neq T.$ In other words $ T $ is not an extreme contraction.\\

\noindent \textbf{Case 2:}
Let $ n=m(m-1) $. If there exists a triplet of distinct numbers $ \{ r, s, t \} \subseteq \{ 1, 2, \dots , n \} $ such that $ | C_r\cap C_s\cap C_t |\geq 2 $ then it follows from the arguments given in the analysis of  \textbf{Case 1} that $ T $ is not an extreme contraction. Now suppose for any triplet of distinct numbers $ \{ r, s, t \} \subseteq \{ 1, 2, \dots , k \} $, $ | C_r\cap C_s\cap C_t|\leq 1 $. According to our assumptions, we have that $ \|T(x_{n+1})\| = 1 $ and $ T(x_{n+1})\not\in E_Y $. Therefore $ T(x_{n+1}) $ has at least two non-zero coordinates, say $b_1,b_2.$ It follows from Lemma \ref{combinatorics} that there exist distinct $k,l \in \{1,2, \ldots, n\}$ such that $ C_k\cap C_l=\{1,2\} $. Without any loss of generality, we assume that $k=1$ and $l=2.$ Let $ \zeta=b_{1}b_{2} $, $ \sigma=a_{11}a_{12} $ and $ \mu=a_{21} a_{22} $. Clearly, $\sigma,\mu,\zeta \neq 0 $ and at least two of them have the same sign. If $\sigma,\mu$ have the same sign, then once again by similar arguments as given in \textbf{Case 1}, it follows that $ T $ is not an extreme contraction. Now we assume that $ \sigma,\zeta $ have the same sign. Without any loss of generality, we assume that $ \sigma,\zeta >0 $. Then we can choose $ \epsilon,\delta  > 0 $ such that the following two conditions hold:\\
$ (i)$ $ \|(a_{11}-\epsilon, a_{12}+\epsilon,\cdots ,a_{1m})\|=\|(a_{11}+\epsilon, a_{12}-\epsilon,\cdots ,a_{1m})\|=\|T(x_1)\|,$\\
$ (ii) $ $ \|(b_1-\delta,b_2+\delta,\cdots ,b_m)\|=\|(b_1+\delta,b_2-\delta,\cdots ,b_m)\|=\|T(x_{n+1})\|. $\\

Let $\lambda = \min\{\epsilon,\delta\}.$ Let $\kappa=\frac{\lambda}{\alpha_1};$  if $|\alpha_1|\geq 1$ and $\kappa=\lambda,$ if $ |\alpha_1| < 1. $  We define two linear operators $T_1,T_2 ~:~ X \rightarrow Y $ by 

$$ T_1(x_1)=(a_{11}-\kappa, a_{12}+\kappa,\cdots ,a_{1m}), \,\,\,\,\,\,\,\, T_2(x_1)=(a_{11}+\kappa, a_{12}-\kappa,\cdots ,a_{1m}), $$
$$ T_1(x_i)=T(x_i)\, \forall~~ i\in \{2, \cdots,n\},  \,\,\,\,\,\,\,\,\,\,\,\,T_2(x_i)=T(x_i)\, \forall~~ i\in \{2, \cdots,n\}. $$

Then,
$$\|T_1(x_{n+1})\|=\|(b_1-\alpha_1\kappa,b_2+\alpha_1\kappa,\cdots ,b_m)\|=\|T(x_{n+1})\|= 1,$$ 
$$\|T_2(x_{n+1})\|=\|(b_1+\alpha_1\kappa,b_2-\alpha_1\kappa,\cdots ,b_m)\|=\|T(x_{n+1})\|= 1.$$
Therefore, $T=\frac{1}{2}(T_1+T_2),$ $\|T_1\|=\|T_2\|=1$ and $T_1,T_2 \neq T.$ Clearly, this contradicts that $T$ is an extreme contraction. In other words, whenever $ T \in L(X,Y) $ is an extreme contraction, we have that $ T(E_{X}) \cap E_{Y} \neq \phi. $ This completes the proof of the theorem.
\end{proof}
In the next example, we illustrate that the assumption $m(m-1) \leq n$ in the above theorem cannot be dropped.

\begin{example} \label{ex2}
Let $X$ be the same Banach space as considered in
Example \ref{ex1} and let $Y=l_{1}^3.$ We define $T:X\rightarrow Y$ in the following way:\\

$$ T(x_1)=(\frac{1}{2},\frac{1}{2},0),\,\,\, T(x_2)=(0,\frac{1}{2},\frac{1}{2}). $$
Therefore, $ T(x_3)=(-\frac{1}{2},0,\frac{1}{2}). $ We claim that $T$ is an extreme contraction. If not, then there exists $T_1,T_2\in S_{L(X,Y)}$ such that $T\neq T_1,T_2$ and $T=\frac{1}{2}(T_1+T_2).$ Let 
$$T_1(x_1)=(\frac{1}{2}+\epsilon_{11},\frac{1}{2}+\epsilon_{12},\epsilon_{13}),\,\,\,\,T_1(x_2)=(\epsilon_{21},\frac{1}{2}+\epsilon_{22},\frac{1}{2}+\epsilon_{23}),$$
where $\epsilon_{ij} $ are arbitrary real numbers for all $i\in\{1,2\}$ and for all $j\in\{1,2,3\}.$ Since $T=\frac{1}{2}(T_1+T_2),$ we must have,
$$T_2(x_1)=(\frac{1}{2}-\epsilon_{11},\frac{1}{2}-\epsilon_{12},-\epsilon_{13}),\,\,\,\,T_2(x_2)=(-\epsilon_{21},\frac{1}{2}-\epsilon_{22},\frac{1}{2}-\epsilon_{23}).$$
From the fact that $\|T_1\|=\|T_2\|=1, $ it is immediate that $ |\epsilon_{11}|,|\epsilon_{12}|,|\epsilon_{22}|,|\epsilon_{23}|\leq \frac{1}{2} $. Therefore, we observe that $ \frac{1}{2} \pm \epsilon_{11}, \frac{1}{2} \pm \epsilon_{12}, \frac{1}{2} \pm \epsilon_{22}, \frac{1}{2} \pm \epsilon_{23} \geq 0. $ Now, we have, $ \| T_1 (x_1) \| = 1 + (\epsilon_{11}+\epsilon_{12}) + |\epsilon_{13}| $ and $ \| T_2 (x_1) \| = 1 - (\epsilon_{11}+\epsilon_{12}) + |\epsilon_{13}| $, which implies that $ (\epsilon_{11} + \epsilon_{12}) = 0, $ i.e., $ \epsilon_{11} = - \epsilon_{12} $ and $ \epsilon_{13} = 0 $. Following similar arguments, we can deduce that $ \epsilon_{22} = -\epsilon_{23} $ and $ \epsilon_{21} = 0 $. Therefore, we have,
$ T_1(x_3)=(-\frac{1}{2}-\epsilon_{11},\epsilon_{11}+\epsilon_{22},\frac{1}{2}-\epsilon_{22}) $ and $ T_2(x_3)=(-\frac{1}{2}+\epsilon_{11},-\epsilon_{11}-\epsilon_{22},\frac{1}{2}+\epsilon_{22}) $. Since $ \|T_1(x_3)\|=\|T_2(x_3) \| \leq 1 $, we have $ | \epsilon_{11}+\epsilon_{22}| \leq (\epsilon_{11}-\epsilon_{22}) $ and $ | \epsilon_{11}+\epsilon_{22}| \leq (\epsilon_{22}-\epsilon_{11}) $. Consequently, $ \epsilon_{11}= \epsilon_{22} = 0 $. This gives us $T_1=T_2=T, $ which is a contradiction. Hence, $T$ is an extreme contraction, although $T(E_X)\cap E_Y=\phi.$ Therefore, the pair $ (X,Y) $ does not have weak L-P property.
\end{example}

We would like to illustrate the applicability of the concept of weak L-P property for a pair of Banach spaces in the study of extreme contractions in some concrete situations. Our next result is oriented towards serving the said goal.

\begin{theorem}
Let $X$ be a $2$-dimensional real Banach space whose unit sphere is a regular hexagon and let $Y=l_{\infty}^2.$ Then $|E_{L(X,Y)}|=36.$
\end{theorem}

\begin{proof}
 Without any loss of generality, we assume that the vertices of $ S_{X} $ are given by $\pm x_1=\pm(1,0),\pm x_2=\pm({\frac{1}{2}},{\frac{\sqrt 3}{2}}),\pm x_3=\pm(-{\frac{1}{2}},{\frac{\sqrt 3}{2}}).$ Let $ T \in L(X, Y) $ be an extreme contraction. Then from Theorem $ 2.2 $ of \cite{SRP}, it follows that $ span(M_T \cap E_X) = X. $ Moreover, Theorem \ref{th1} of the present paper implies that there exists atleast one $ x_i \in E_X, $ $ i \in \{1,2,3\}, $ such that $ T(x_i) \in E_Y. $ Now, there are two possibilities:\\
 
{\bf  (1)  $ | M_T\cap E_X | = 4, $} \,  {\bf  (2)  $ | M_T\cap E_X | = 6 .$}\\

Using Theorem $ 2.2 $ of \cite{SRP}, it follows that $ T( M_T \cap E_X ) \subseteq E_Y $ in the first possibility. We subdivide possibility {\bf (1) } into three possible cases:\\

{\bf (i)  $ x_1, x_2 \in M_T, $  (ii)  $ x_2, x_3 \in M_T, $  (iii)  $ x_1, x_3 \in M_T $}. \\

We note that if $ a,b \in E_Y $ are distinct then $ \| a-b \| = 2 $. Since $ x_3 = x_2- x_1 $, it follows that $ T(x_1)=T(x_2) $ in case \textbf{(i)}.  In that case it is easy to see that $ 4 $ such extreme contractions are possible. Also, $ x_1 = x_2 - x_3 $ and therefore similar conclusion holds true for case \textbf{(ii)}. Now we consider case \textbf{(iii)}. Since $ x_3 + x_1 = x_2 $, clearly $ T(x_3)\neq T(x_1) $. Moreover, $ T(x_1) $ and $ T(x_3) $ cannot be linearly independent since for any pair of linearly independent vectors $ \{ a,b \} \subseteq E_Y $, $ \| a+b \| = 2 $. Therefore, the only possibility is $ T (x_1) = - T(x_3), $ which provides $ 4 $  choices for $ T(x_1) $. Therefore, we obtain exactly $ 12 $ extreme contractions from possibility \textbf{(1)}.\\
Now, suppose $ | M_T\cap E_X | = 6 $.  Without any loss of generality, suppose $ T(x_1) = (1,1) $ and then there are the following $ 4 $ possibilities for $ Tx_2: $\\

{\bf(a) $ T(x_2) = (1,t), $ (b) $ T(x_2) = (-1,t), $ (c) $ T(x_2) = (t,1), $ (d) $ T(x_2) = (t,-1) $},  where $ t \in [-1,1]. $\\

Let us separately consider all the above possibilities:\\

\textbf{(a)} Let $ T(x_2) = (1,t), $ where $ t \in [-1,1]. $ Then $ T(x_3) = T(x_2)-T(x_1) = (0,t-1), $ as $ x_3=x_2-x_1. $ Therefore, $ \|Tx_3\| > 1 $ for all $ t \in [-1,0). $ This is a contradiction to the fact that $ \|T\|=1. $ Also, $ \|T(x_3)\| < 1, $ for all $ t \in (0,1]. $ This contradicts our assumption that $ x_3 \in M_T. $ We therefore deduce that the only possibility is $ t=0, $ i.e., $ T(x_2) =(1,0). $ Then $ T(x_3)=(0,-1). $ Now, we claim that $ T \in L(X,Y), $ given by $ T(x_1)=(1,1), ~ T(x_2)=(1,0) $ is an extreme contraction. If not, then there exist $ T_1, T_2 \in S_{L(X,Y)} $ such that $ T_1, T_2 \neq T $ and $ T= \frac{1}{2}(T_1+T_2).$ Suppose, 
$$ T_1(x_1)= (1+ \epsilon_{11},1+ \epsilon_{12}), \,\,\, T_1(x_2)= (1+\epsilon_{21},\epsilon_{22}), $$ where $\epsilon_{ij} $ are arbitrary real numbers for all $i,j\in\{1,2\}.$ Then,
$$ T_2(x_1)= (1- \epsilon_{11},1-\epsilon_{12}), \,\,\, T_2(x_2)= (1-\epsilon_{21},-\epsilon_{22}), $$ 
as otherwise, $ T = \frac{1}{2}(T_1+T_2) $ is not possible. Since $ T_1, T_2 \in S_{L(X,Y)}, $ it is easy to see that $ \epsilon_{11}=\epsilon_{12}=\epsilon_{21}=0. $ Hence, $ T_1(x_1)= (1,1), ~~~ T_1(x_2)= (1,\epsilon_{22}), $ and $ T_2(x_1)= (1,1), ~~~ T_2(x_2)= (1,-\epsilon_{22}). $ Then $T_1(x_3)=(0, \epsilon_{22}-1) $ and $ T_2(x_3)=(0,-\epsilon_{22}-1). $ If $ \epsilon_{22} > 0, $ then $ \|T_2(x_3)\|=1+\epsilon_{22}>1 $ and if $ \epsilon_{22} < 0, $ then $ \|T_1(x_3)\|=1-\epsilon_{22}>1. $ In both cases, we get a contradiction to the fact that $ \|T_1\|=\|T_2\|=1. $ Therefore $ \epsilon_{22} = 0 $ and $ T $ must be an extreme contraction.\\

\textbf{(b)} Let $ T(x_2) = (-1,t), $ where $ t \in [-1,1]. $ Then $ T(x_3) = T(x_2)-T(x_1) = (-2,t-1), $ as $ x_3=x_2-x_1. $ Therefore, $ \|T(x_3)\| = 2 $ for all $ t \in [-1,1], $ which is a contradiction to the fact that $ \|T\|=1. $ So, this case is not possible.\\

{\bf (c) } Proceeding in the similar way like $ {\bf (a) } $, we can prove that $ T $ is an extreme contraction only when $ t = 0 $, i.e., $ T(x_2)=(0,1) $.\\

{\bf (d) } This case is similar to case {\bf (b)}.\\

 \textbf{Conclusion:} If $ T(x_1)=(1,1) $ and $ x_2, x_3 \in M_T, $ then there are $ 2 $ extreme contractions. Similarly, if $ T(x_1) $ is any one of the other three extreme points of $ B_Y $ and $ x_2, x_3 \in M_T, $ we get $ 2 $ other extreme contractions in each cases. Thus, if $ T(x_1) $ is extreme point of $ B_Y $ and $ x_2, x_3 \in M_T, $ there are exactly $ 8 $ extreme contractions. Similar argument is applicable if we assume $ T(x_2) $ is extreme point of $ B_Y $ and $ x_1, x_3 \in M_T. $ Indeed, in this case, there are also exactly $ 8 $ extreme contractions. The same happens when $ T(x_3) $ is extreme point of $ B_Y $ and $ x_1, x_2 \in M_T. $  Therefore, we obtain a total of $ 24 $ extreme contractions from possibility \textbf{(2)}. Now, combining possibilities \textbf{(1)} and \textbf{(2)}, we conclude that there are exactly $ 36 $ extreme contractions in $ L(X, Y). $ This establishes the theorem.
\end{proof}

The following example illuminates that the number of extreme points of the unit ball of the domain space plays a vital role in determining the weak L-P property for a pair of Banach spaces. In particular, it shows that Theorem \ref{th1} no longer holds true if the number of extreme points of the unit ball of the domain space is strictly greater than $ 2n+2, $ where $ n $ is the dimension of $ X. $

\begin{example}\label{ex3}
Let $X$ be a two dimensional real Banach space whose unit sphere is a regular octagon with vertices, 
$$\pm x_1= \pm(1,0),\pm x_2= \pm(\frac{1}{\sqrt 2},\frac{1}{\sqrt 2}),\pm x_3= \pm(0,1),\pm x_4= \pm(-\frac{1}{\sqrt 2},\frac{1}{\sqrt 2})$$ and let $Y=l_{\infty}^2.$
Clearly, $\{x_1,x_2\}$ forms a basis for $X.$ Also, $x_3=\sqrt {2}x_2-x_1$ and $x_4= x_2-\sqrt {2}x_1.$ Now, we define $T:X\rightarrow Y$ in the following way:\\

$$T(x_1)=(1,\sqrt {2}-1),\,\,\,\,T(x_2)=(\sqrt {2}-1,1)$$

Therefore,
$$T(x_3)=(-\sqrt {2}+1,1),\,\,\,\,T(x_4)=(-1,\sqrt {2}-1).$$
We claim that $T$ is an extreme contraction. If not, then there exist $T_1,T_2\in S_{L(X,Y)}$ such that $ T_1, T_2 \neq T $ and $T=\frac{1}{2}(T_1+T_2).$ Suppose
$$T_1(x_1)=(1+\epsilon_{11},\sqrt {2}-1+\epsilon_{12}),\,\,\,\,T_1(x_2)=(\sqrt {2}-1+\epsilon_{21},1+\epsilon_{22}),$$
where $\epsilon_{ij} $ are arbitrary real numbers for all $i,j\in\{1,2\}.$ Then
$$T_2(x_1)=T(1,0)=(1-\epsilon_{11},\sqrt {2}-1-\epsilon_{12}),$$
$$T_2(x_2)=T(\frac{1}{\sqrt 2},\frac{1}{\sqrt 2})=(\sqrt {2}-1-\epsilon_{21},1-\epsilon_{22}),$$ 
as otherwise, $ T=\frac{1}{2}(T_1+T_2) $ is not possible. Since $ \| T_1 \| = \| T_2 \| = 1, $ it is easy to see that $\epsilon_{11},\epsilon_{22}=0.$ Hence,
$$T_1(x_1)=T_1(1,0)=(1,\sqrt {2}-1+\epsilon_{12}),\,\,\,\,T_1(x_2)=T_1(\frac{1}{\sqrt 2},\frac{1}{\sqrt 2})=(\sqrt {2}-1+\epsilon_{21},1).$$

Therefore, $ T_1(x_3)=\sqrt {2}T_1(x_2)-T_1(x_1)=(1-\sqrt{2}+\sqrt{2}\epsilon_{21},1-\epsilon_{12}) $ and $ T_1(x_4)=T_1(x_2)-\sqrt {2}T_1(x_1)=(-1+\epsilon_{21},-1+\sqrt{2}-\sqrt{2}\epsilon_{12}) .$
Similarly,
$$T_2(x_1)=T_2(1,0)=(1,\sqrt {2}-1-\epsilon_{12}), \,\,\,\, T_2(x_2)=T_2(\frac{1}{\sqrt 2},\frac{1}{\sqrt 2})=(\sqrt {2}-1-\epsilon_{21},1).$$

Therefore, $ T_2(x_3)=\sqrt {2}T_2(x_2) - T_2(x_1)=(1-\sqrt{2}-\sqrt{2}\epsilon_{21},1+\epsilon_{12}) $ and $ T_2(x_4)=T_2(x_2)-\sqrt {2}T_2(x_1)=(-1-\epsilon_{21},-1+\sqrt{2}+\sqrt{2}\epsilon_{12}). $
Now, if $\epsilon_{12}> 0$ then $\|T_2(x_3)\| \geq 1+\epsilon_{12}>1$ and if $\epsilon_{12}<0$ then $\|T_1(x_3)\| \geq 1-\epsilon_{12}>1,$ which contradicts our assumption that $\|T_1\|=\|T_2\|=1.$ Therefore, $\epsilon_{12}=0.$ Similarly, $\epsilon_{21}=0.$ Thus, $T_1=T_2=T, $ which is a contradiction. This proves that $T$ is an extreme contraction but $ T(E_X)\cap E_Y = \phi. $ Hence, the pair $ (X,Y) $ does not have weak L-P property.
\end{example}

In view of the results obtained in the present article, it is perhaps appropriate to end the article with the following open question:\\

\textbf{Open question:} Let $ X $ and $ Y $ be real Banach spaces. Obtain a necessary and sufficient condition on $ X $ and $ Y $ for the pair $ (X,Y) $ to have weak L-P property.

\bibliographystyle{amsplain}

\end{document}